\documentclass[12pt]{article}

\usepackage{diagbox}
\usepackage{mathtools}
\usepackage{latexsym}
\usepackage{epsfig}
\usepackage{amsmath,amsthm,amssymb,amsfonts,enumerate}

\usepackage[a-1b]{pdfx}
\usepackage{hyperref}

\usepackage{color}

\def\cH{{\mathcal H}}
\def\nn{\nonumber}
   \def\D{\Delta}
\def\e{\varepsilon}    
  
     \def\l{\lambda}

\newtheorem{theorem}{Theorem}
\newtheorem{conjecture}{Conjecture}
\newtheorem{lemma}[theorem]{Lemma}

\newtheorem{claim}{Claim}

\newcommand{\rdup}[1]{{\left\lceil #1\right\rceil }}
\newcommand{\rdown}[1]{{\left\lfloor #1\right \rfloor}}

\newcommand{\brac}[1]{\left(#1\right)}

\newcommand{\bfrac}[2]{\left(\frac{#1}{#2}\right)}

\def\cE{{\cal E}}

\newcommand{\set}[1]{\left\{#1\right\}}
\def\sm{\setminus}

\def\E{\mathbb{E}}

\def\Pr{\mathbb{P}}

\def\cF{{\cal F}}
\newcommand{\ignore}[1]{}

\def\cB{{\mathcal B}}

\def\cE{{\mathcal E}}
\def\cF{{\mathcal F}}
\def\cG{{\mathcal G}}
\def\cH{{\mathcal H}}

\def\cM{{\mathcal M}}

\def\cS{{\mathcal S}}

\def\cX{{\mathcal X}}

\newcommand{\beq}[2]{\begin{equation}\label{#1}#2\end{equation}}

\def\nn{\nonumber}

\def\cG{\mathcal{G}}

\def\imax{i_{\max}}

\usepackage{tikz}
\usetikzlibrary{decorations.pathmorphing}
\usetikzlibrary{positioning}
\usetikzlibrary{arrows,automata}
\usetikzlibrary{shapes.misc}
\usetikzlibrary{backgrounds}
\usetikzlibrary{arrows,shapes}

\begin{document}

\date{}
\title{Binomial Random Matroids}
\author{Patrick Bennett\thanks{Department of Mathematics, Western Michigan University, Kalamazoo MI 49008-5248, Research supported in part by Simons Foundation Grant \#848648.}\and Alan Frieze\thanks{Department of Mathematical Sciences, Carnegie Mellon University, Pittsburgh PA15213, Research supported in part by NSF grant DMS2341774.}}
\maketitle
\begin{abstract}
Let $\mathcal B=\mathcal B_{k,n,p}$ be a random collection of $k$-subsets of $[n]$ where each possible set is present independently with probability $p$. Let $\cal E_{\mathcal B}$ be the event that $\mathcal B$ defines the set of bases of a matroid. We prove that
If $p= 1-\frac{c_n}{(k(n-k)\binom nk)^{1/2}}$ where $0\leq c_n\leq \infty$, then 
 \[
 \lim_{n\to\infty}\Pr[\cal E_{\cal B}\mid |\cal B|\geq2]=\begin{cases}1&c_n\to0.\\e^{-c^2/2}&c_n\to c.\\0&c_n\to \infty.\end{cases}\]
 In addition, we identify a condition preventing the occurence of $\cal E_{\cal B}$ and prove a hitting time version for the occurence of $\cal B$. We also prove that when $\cal E_{\mathcal B}$ occurs, $\mathcal B$ defines a sparse paving matroid w.h.p.
 
 In addition, study a greedy algorithm that produces a random matroid defined by a collection of hyperplanes. We use this to improve the estimates in \cite{HPV} on $\log m(n,k),\log p(n,k), \log s(n,k)$ where $ m(n, k), p(n, k), s(n, k)$ denote the number of matroids, paving matroids, and sparse paving matroids (respectively) of rank $k$ on $[n]$. Our improvement lies in that we can deal with $k$ growing slowly with $n$ as opposed to $k=O(1)$ in \cite{HPV}. More generally, we obtain estimates for the number of matchings in nearly-regular hypergraphs with small codegree, which may be of independent interest.
\end{abstract}

\section{Introduction}

We refer the unfamiliar reader to Oxley's book \cite{Oxley} for basic definitions and facts about matroids.  Let $\cB$ be a family of $k$-subsets of $[n]$. Recall that $\cB$ is the set of {\em bases} for a matroid if $\cB \neq 0$ and we have the {\em basis exchange property}, i.e. 
\beq{exchange}{
\text{for all } B_1, B_2 \in \cB, x \in B_1 \setminus B_2, \text{ there exists } y \in B_2 \setminus B_1 \text{ such that } B_1 \setminus \{x\} \cup \{y\} \in \cB.
}
The associated matroid $M=M(\cB)$ can be viewed as being defined by $\cB$. We say that $M$ has rank $k$ (the cardinality of the bases). A {\em circuit} of $M$ is a minimal set $C \subseteq [n]$ which is not contained in any basis. We say that a rank $k$ matroid $M$ is a {\em paving matroid} if every circuit has cardinality $k$ or $k+1$. Welsh \cite{Welsh} asked whether most matroids were paving, and Mayhew, Newman, Welsh and Whittle \cite{MNW} went further to conjecture that almost all matroids are paving.

\begin{conjecture}[\cite{MNW}]\label{conj:paving}
    As $n\rightarrow \infty$, the proportion of matroids on $[n]$ that are paving tends to 1. 
\end{conjecture}

For our convenience, we will define the following whether $\cB$ defines a matroid or not. We will say $\cB$ has the {\em paving property} if we have $|C| \in \{k, k+1\}$ for every $C \subseteq [n]$ which is minimal with respect to not  being contained in any $B \in \cB$. Of course, if $\cB$ does define a matroid and has this paving property, it is a paving matroid. 

Let $\cB=\cB_{n,k,p}$ be a random collection of $k$-subsets of $[n]$ where each possible set is present independently with probability $p$. Let $\cE_\cB$ be the event that $\cB$ is the set of bases of a matroid. Let $\cF_\cB$ be the event $\set{\exists A_1,A_2\in \binom{[n]}{k}\sm\cB: |B\cap B'|=k-1}$. Note that if  \eqref{exchange} fails for $x,B_1,B_2$ then $\cF_\cB$ holds. Indeed, we must have $|B_2\sm B_1|\geq 2$. So, let $x_1,x_2\in B_2\sm B_1$ and let $A_i=(B_1\cup\set{x_i})\sm\set{x},i=1,2$. Then $|A_1\cap A_2|=k-1$ and $A_1,A_2\notin \cB$. 

Thus, we have two conditions under which $\cE_\cB$ trivially holds: the first condition is when $|\cB|=1$, and the second is when $\cF_\cB$ fails. Our next theorem essentially says that when $\cB$ is random, these two conditions are typically necessary for $\cE_\cB$ to hold. The bounds on $k$ is a consequence of our proof.  It should most likely hold for all $2 \le k \le n-2$.
\begin{theorem}\label{th1}
Let $7\leq k\leq n-7$. 
\begin{enumerate}[(a)]
 \item \label{th1b} Suppose that $B_1,B_2,\ldots,B_{\binom{n}{k}}$ is a random ordering of $\binom{[n]}{k}$. Let $\cB_t=\set{B_1,B_2,\ldots,B_t}$ and let $\tau$ be the smallest $t>1$ such that $\neg \cE'_{\cB_t}$. Then w.h.p.~$\cE_{\cB_t}$ fails for all $2\leq t<\tau$. (Of course, $\cE_{\cB_t}$ also fails for $t=0$, and holds for $t=1$ and $t \ge \tau$). Also, w.h.p.~$\cB_{\tau}$ defines a paving matroid. 
\item \label{th1a} If $p= 1-\frac{c_n}{\left(k(n-k)\binom nk \right)^{1/2}}$ where $0\leq c_n< \infty$, then 
 \beq{11}{
 \lim_{n\to\infty}\Pr[\cE_\cB\mid |\cB|\geq2]=\begin{cases}1&c_n\to0.\\e^{-c^2/2}&c_n\to c.\\0&c_n\to \infty.\end{cases}}
Furthermore, when $c_n \rightarrow 0$ or $c$, $\cB$ has the paving property w.h.p.
 \end{enumerate}
\end{theorem}
{\bf Remark:} We discuss the conditioning on $|\cB|\ge 2$ in part (\ref{th1a}) above. Of course, if $|\cB|=1$ then $\cE_\cB$ trivially holds. Note that $|\cB|\geq 2$ w.h.p.~whenever $p = \omega(1/\binom nk)$ which holds in the first two cases, i.e.~ when $c_n$ approaches $0$ or $c$. When $p=b/\binom{n}{k}$ for some constant $b>0$, $|\cB|$ will be asymptotically distributed as the Poisson random variable with mean $b$. This means that there is a probability $\sim be^{-b}$ that $|\cB|=1$.

Of course Theorem \ref{th1} brings us no closer to resolving Conjecture \ref{conj:paving}. However, Theorem \ref{th1} can be interpreted as weak evidence for the conjecture. Indeed, Theorem \ref{th1} essentially says that a random $\cB$ has the paving property whenever it is likely to form a matroid.

A matroid can alternatively be defined by its collection of {\em hyperplanes}. Suppose we have a collection $\cH$ of subsets of $[n]$ such that 
\begin{enumerate}[a)]
    \item $H_1 \not \subset H_2$ for all $H_1, H_2 \in \cH$,
    \item for all distinct $H_1, H_2 \in \cH$ and $x \notin H_1 \cup H_2$, $(H_1 \cap H_2)\cup \{x\}$ is contained in some $H_3 \in \cH$. 
\end{enumerate}
Then we say $\cH$ is the collection of hyperplanes for a matroid. We say this matroid is {\em sparse paving} if every hyperplane has cardinality $k-1$ or $k$, and every $(k-1)$-subset of $[n]$ is contained in a unique hyperplane.
Let $ m(n, k), p(n, k), s(n, k)$ denote the number of matroids, paving matroids, and sparse paving matroids (respectively) of rank $k$ on $[n]$. Our next theorems extend results of  van der Hofstad,  Pendavingh and van der Pol \cite{HPV}, who proved the same estimates in the case where $k$ is constant.

\begin{theorem}\label{th2}
    Let $k=k(n) \ge 3$ and $k=o(\log n)$. Then  
    \[
   \log s(n, k) = \frac 1n \binom nk \Big( \log n+1 - k +o(1) \Big).
    \]
\end{theorem}

\begin{theorem}\label{th3}
    Let $k=k(n) \ge 4$ and $k=o\brac{\frac{\log^{1/2}n}{\log \log n}}$. Then  $\log p(n, k) $ and $\log m(n, k)$ are both 
    \[
   \frac 1n \binom nk \Big( \log n+1 - k +o(1) \Big).
    \]
\end{theorem}
Theorem \ref{th3} also has a flavor that feels relevant to Conjecture \ref{conj:paving}, which could be restated as $\sum_{k} p(n, k) =(1+o(1))\sum_k m(n, k)$. Note that Theorem \ref{th3} requires $k \ge 4$. Interestingly, Kwan, Sah and Sawhney \cite{KSS} proved that Theorem \ref{th3} could not be extended to $k=3$ since actually $p( n, 3)$ is much larger than $ s(n, 3)$ (which is estimated by Theorem \ref{th2}). 

Theorem \ref{th3} will follow easily from Theorem \ref{th2} and some facts already proved by van der Hofstad,  Pendavingh and van der Pol \cite{HPV}. We prove Theorem \ref{th2} by adapting the tools and proof techniques used in \cite{HPV}. In particular, we use methods for matchings in hypergraphs. Suppose $\cH$ is the collection of hyperplanes for a sparse paving matroid, and let $\cH_k \subseteq \cH$ be the hyperplanes of size $k$. Then from the definition of sparse paving matroids, no two hyperplanes in $\cH_k$ intersect in $k-1$ elements. Note that the hyperplanes of size $k-1$ are just all the $(k-1)$-sets which are not contained in a hyperplane of size $k$. (A $(k-1)$-set must be contained in at least one hyperplane. Thus $\cH_k$ is a partial Steiner system $S_p(n, k, k-1)$ i.e. a set of $k$-subsets, such that each $(k-1)$-set is contained in at most one such $k$-set.)
 Thus, there is a one-to-one correspondence between partial Steiner systems $S_p(n, k, k-1)$ and sparse paving matroids of rank $k$. Meanwhile, there is also a one-to-one correspondence between partial Steiner systems $S_p(n, k, k-1)$ and matchings in an appropriately defined hypergraph. In particular, for $\ell \le k \le n$ we let $\cG_{n, k, \ell}$ be the hypergraph having vertex set $\binom{[n]}{\ell}$ and a hyperedge $\binom{K}{\ell}$ for every $k$-subset $K \subseteq [n]$. Then a partial Steiner system $S_p(n, k, \ell)$ is just a matching in $\cG_{n, k, \ell}$.

Thus, we can prove Theorem \ref{th2} by estimating the number of matchings in a hypergraph. For the case of constant $k$,  the lower bound in \cite{HPV} is proved by applying a result of the first author and Bohman \cite{BB} who analyzed a random process for matchings in $k$-uniform hypergraphs. To prove the lower bound in Theorem \ref{th2} we will need to extend the result of \cite{BB} to allow for growing $k=k(n)$. 

Our next theorem may be of independent interest since matchings in hypergraphs are widely applicable. Here we will only directly use the second part which estimates the number of matchings in a hypergraph. However to prove the second part we will use the first part, which gives sufficient conditions for an almost-perfect matching. This theorem also extends results of Chakravarty and Spanier \cite{CS}, who analyzed a random greedy process for producing linear $q$-uniform hypergraphs, i.e. partial Steiner systems $S_p(n, q, 2)$ for $q = o(\log^{1/2} n)$. 

\begin{theorem}\label{th4}
     Let $\cH$ be a $k$-uniform hypergraph $D$-regular hypergraph on $N$ vertices such that every pair of vertices is in at most $D\phi$ edges together. Suppose that as $N \rightarrow \infty$ we have 
     \beq{eqn:th4assumption}
     {
     \phi = o\brac{\log^{-10}N}, \qquad k=o\brac{\log \brac{\phi^{-1}}}.
     }
    Then 
    \begin{enumerate}[a)]
        \item \label{th4a}  $\cH$ has a matching $\cM$ of size at least
    \beq{eqn:matchingsize}
    {
    \brac{1 - \phi^{\frac{1}{100k}}}\frac Nk,
    }  
        \item \label{th4b} the number of matchings in $\cH$ (including non-perfect matchings) is 
  \beq{eqn:matchingcount}
  {
    \bigg((1+o(1)) De^{1-k}\bigg)^{N/k}.
}
Note that here we are stating both an upper and lower bound, which agree up to the ``$o(1)$'' term.
    \end{enumerate}
\end{theorem}
{\bf Remark:} The assumptions of the theorem above imply some additional bounds on the parameters. Note that if we fix $v$ and sum over $u$ the number of edges containing both $u, v$ we get $(k-1)D \le (N-1)D\phi$.
Therefore
\beq{eqn:phibound}{
\phi \ge \frac{k-1}{N-1},  \qquad k = o\brac{\log N} ,
}
where the second inequality follows from the first and \eqref{eqn:th4assumption}.

Part \ref{th4a}) of Theorem \ref{th4}, as well as the lower bound in part \ref{th4b}), follow from analyzing a random greedy process. The upper bound in part \ref{th4b}) follows from an entropy argument. There is significant history of using entropy to bound the number of matchings, including Radhakrishnan's proof of Bregman's theorem on counting matchings in bipartite graphs \cite{R}, Cutler and Radcliffe's proof of the Kahn-Lov\'asz theorem for matchings in general graphs \cite{CR}, and Lineal and Luria's bound on the number of Steiner systems $S(n, 3, 2)$ \cite{LL}.

\section{Proof of Theorem \ref{th1} }
\begin{proof}
We start with part (\ref{th1b}). Let $N=\binom{n}{k}$. Given that $\cF_\cB$ is monotone decreasing and that $\neg\cF_\cB$ implies $\cE_{\cB}$ deterministically, we only have to show that $\neg\cE_{B_t}$ occurs w.h.p. for $2\leq t<\tau$. 

Suppose the first two edges of the process intersect in $|B_1 \cap B_2|=j$ vertices. W.h.p. $j \le k-2$, and so starting at step $2$ we have a ``problem'' to fix, i.e.~a violation of \eqref{exchange}. To fix this problem we need some edge of the form  $B_1 - x + y$ for every $x \in B_1 \sm B_2$, where $y \in B_2 \sm B_1$. The number of ways to choose a $y$ for each $x$ is $(k-j)^{k-j}.$ Now to fix the problem we need to have all the edges $B_1 - x + y$ (for the selected pairs $x, y$). Let 
\[
M=\rdown{n^3N^{1/2}}. 
\]
The probability of getting all $k-j$ of these edges by step $t_1=M+1$ is at most $(k-j)^{k-j} (M/N)^{k-j}$, which we claim is $o(1)$. Indeed, if $k = o(n^{1/2})$ then actually we have that w.h.p.~$j=0$, while $M/N = O(n^{-1/2})$ since $k \ge 7$. Meanwhile for larger $k$, say for all $k$ from $n^{1/3}$ to $n-7$, $M/N \rightarrow 0$ faster than any constant power of $n$.

Let $Z_t$ be the number of (ordered) pairs $(A_1, A_2)\in \binom{[n]}{k}^2$ such that $|A_1\cap A_2|=k-1$ and $A_1,A_2\notin \cB$. For the analysis from $t_1$ onwards, we use the second moment of $Z_t$ to put a bound on $Z_t$. Letting $p_t=t/N$, we have
\begin{align*}
\E[Z_t]&=Nk(n-k)(1-p_t)^2.\\
\E[Z_t(Z_t-1)]&=\sum_{\substack{S,S'\in \binom{[n]}{k+1}\\|S\sm S'|>2}}\brac{k(k+1)(1-p_t)^{2}}^2+X_2\leq \E(Z_t)^2+X_2,
\end{align*}
where 
\[
X_2\leq \E(Z_t)((n-k)^2(1-p_t)^2+2(n-k)(1-p_t)).
\]
{\bf Explanation:} In our expression for $\E[Z_t(Z_t-1)]$, the first sum takes care of all pairs $(A_1, A_2), (A_1', A_2')$ such that $A_1 \cup A_2 = S, A_1' \cup A_2' = S'$ and $S \sm S' >2$. In that case, $A_1, A_2, A_1', A_2'$ are all distinct. $X_2$ is the sum of the remaining terms.  $X_2$ is at most $\E(Z_t)$ times the sum of two terms. The first is the sum over sets $S\in\binom{[n]}{k+1}$ and $x_1,x_2\in S\sm [k+1]$ for which $|S\cap[k+1]|=k-2$ and $S\sm\set{x_i}\notin\cB,i=1,2$ holds. The second is twice the sum  over sets $S\in \binom{[n]}{k+1}$ such that $S=[k-1]\cup\set{x},x>k+1$. In the first case we must find two distinct sets not in $\cB$, neither of which is contained in $[k+1]$ and the in the second case we must find one set. It follows from the Chebyshev inequality that
\beq{halfEZ}{
\Pr(Z_t\leq \E(Z_t)/2)\leq \frac{4(\E(Z_t)+X_2)}{\E(Z_t)^2}\leq \frac{4}{\E(Z_t)}+\frac{4(n-k)}{kN}+\frac{8}{k(1-p_t)N}.
}
Now, assuming that $Z_t\geq \E(Z_t)/2$, there will be a set of at least $\E(Z_t)/2n^4$ triples $S\in \binom{[n]}{k+1},x_1,x_2\in S$ such that $S \sm \{x_i\} \notin \cB$ and $|S\sm S’|>2$ for each pair of distinct triples $(S, x_1, x_2), (S', x_1', x_2')$. Condition on the existence of a fixed such set of  $\E(Z_t)/2$ such triples.

Fix one of these triples and let $T=S\sm\set{x_1,x_2}$. If $\cE_{\cB}$ occurs then we cannot find $y_1\in T,y_2\in [n]\sm S$ such that $B_1=T\cup\set{y_2}\in \cB$ and $B_2=(T\sm\set{y_1})\cup\set{x_1,x_2}\in\cB$. The existence of $y_1,y_2$ and the exchange property would imply that either $T\cup\set{x_1}\in\cB$ or $T\cup\set{x_2}\in\cB$ is a replacement for $y_2$ in $B_1$, contradiction. So, either there is no $y_1$ or there is no $y_2$. Note that neither of $B_1\notin\cB$ nor $B_2\not\in\cB$ are part of the conditoining. Also, the $k$-sets involved will be distinct for different $S$. It follows that 
\beq{B|A1}{
\Pr(\cE_{\cB_t})\leq \frac{4}{\E(Z_t)}+\frac{4(n-k)}{kN}+\frac{8}{k(1-p_t)N}+(1-p_t)^{\E(Z_t)/2n^4}
}
Let $i_0=\rdown{\frac{N}{M}}-1$ and let $t_i=iM+1$ for $0\leq i\leq i_0$. Let $Z_{t_i}$ denote $Z$ evaluated at $t_i$. 
\beq{Zi}{
\E(Z_{t_i})\sim Nk(n-k)\bfrac{N-iM+1}{N}^2\geq \frac{k(n-k)(N-iM)^2}{N}.
}
We have dealt with $t\leq t_1$ and we have from \eqref{halfEZ} and \eqref{B|A1} that 
\begin{align*}
&\Pr(\exists 1\leq i\leq i_0:\cE_{\cB_{t_i}}\mid t_{i_0}\leq \tau)\\
&\leq \sum_{i=1}^{i_0}\brac{\frac{4}{\E(Z_{t_i})}+\frac{4(n-k)}{kN}+\frac{8}{k(1-p_t)N}+(1-p)^{\E(Z_{t_i})/n^4}}\\
&\leq  o(1)+\sum_{i=1}^{i_0} \brac{\frac{4N}{k(n-k)(N-iM)^2}+\frac{8}{k\brac{N-iM}}+\bfrac{N-iM}{N}^{k(n-k)(N-iM)^2/(n^3N)}}.
\end{align*}
Now
\[
\sum_{i=1}^{i_0}\frac{N}{(N-iM)^2}=\frac{N}{M^2}\sum_{i=1}^{i_0}\frac{1}{\brac{\frac{N}{M}-i}^2}\leq\frac{2N}{M^2}\sum_{i=1}^{i_0}\frac{1}{(i_0+1-i)^2}\leq \frac{2N}{M^2}\sum_{i=1}^{i_0}\frac{1}{i^2}=O(n^{-6}),
\]
and
\[
\sum_{i=1}^{i_0}\frac{1}{N-iM}=\frac{1}{M}\sum_{i=1}^{i_0}\frac{1}{\frac{N}{M}-i}\leq \frac{2}{M}\sum_{i=1}^{i_0}\frac{1}{i}=O\brac{\frac{\log N}{n^2N^{1/2}}},
\]
and
\begin{align*}
\sum_{i=1}^{i_0} \bfrac{N-iM}{N}^{k(n-k)(N-iM)^2/(n^4N)}&\leq \sum_{i=1}^{i_0}\exp\set{-\frac{iMk(n-k)}{n^4}\brac{1-\frac{iM}{N}}^2}\\
&\leq \sum_{i=1}^{i_0}\exp\set{-\frac{iNk(n-k)}{n^4}\cdot\frac{M^2}{N^2}}\leq  \sum_{i=1}^{i_0}e^{-ikn^2(n-k)}=o(1).
\end{align*}
 And so
\beq{ti}{
\Pr(\exists 1\leq i\leq i_0:\cE_{\cB_{t_i}}\mid t_{i_0}\leq \tau)=o(1).
}
This means that w.h.p. there exist, for every $i\leq i_0$, a pair of sets $B,B’$ for which \eqref{exchange} does not hold. Therefore there are at least 2 ``missing'' $k$-sets w.r.t. $B,B’$ and the probability they are both chosen by the time $t_{i+1}$ is $O(M\cdot (n/(N-iM)^2)$. So, 
\beq{before}{
\Pr(\exists t_1\leq t\leq t_{i_0}:\cE_{\cB_t})=o(1)+O\brac{\sum_{i=1}^{i_0}\frac{n^2M}{\brac{N-iM}^2}}=o(1)+O\bfrac{n^2M }{N^2-i_0MN}=o(1)+O\bfrac{n^2}{N}=o(1).
}
We show next that for $t\geq t_{i_0}$,
\beq{B|A}{
\Pr(\cE_\cB\mid\cF_\cB)\leq \binom{n}{k+2}  (k+2)^3 (1-p_{t_{i_0}})^3+n^2\brac{1-\frac{k^2}{n^2}}^{t_{i_0}}\leq 2Nk(n-k)^2(1-p_{t_{i_0}})^3.
}
First we observe that if $\cF_1=\set{\exists x_1,x_2\in [n]:\not\exists t\leq t_{i_0}\ s.t.\ x_1,x_2\in B_t}$ then
\beq{x1x2}{
\Pr(\cF_1)\leq n^2\brac{1-\frac{k^2}{n^2}}^{t_{i_0}}. 
}
Then if $Y_t$ is the number of triples $(A_1, A_2, A_3)$ of non-edges with $|A_1 \cup A_2 \cup A_3|\le k+2$ and $\cF_2=\set{\exists t\geq t_{i_0}:Y_t>0}$, then
\beq{EY}{
\Pr(\cF_2)\leq \binom{n}{k+2}  (k+2)^3 \bfrac{N-t_{i_0}}{N}^3.
}
Suppose now that $t<\tau$ and so there exist $A\in\binom{[n]}{k-1},x_1,x_2\in[n]$ such that $A_i=A\cup\set{x_i}, i=1,2\in \binom{[n]}{k}\sm\cB$. Assume that $\cF_1,\cF_2$ do not occur. $\neg\cF_1$ implies that there there exists $B\in\cB$ containing $x_1,x_2$. Choose $B_1$ to maximise $|A\cap B|$ among all such choices. Suppose that we can find distinct $y_1,y_2,y_3\in A\sm B_1$. Now choose $y\in B_1\sm A$. Now at least one of $(B_1\cup \set{y_i})\sm \set{y},i=1,2,3$ must be in $\cB$, else $\cF_2$ occurs. So we can assume that $|A\sm B_1|\leq2$. If $B_1=(A\cup\set{x,x_1,x_2})\sm \set{y_1,y_2}$ say then $\neg\cF_2$ implies that $B_{2}=(B_1\cup\set{y_1}) \sm\set{x}=(A\cup\set{x_1,x_2})\sm \set{y_2}\in \cB$. $\neg\cF_2$ implies that $B_z=A\cup\set{z}\in\cB$ for all $z\notin A\cup\set{x_1,x_2}$. Fix some such $z$. Then if $\cE_\cB$ occurs, one of $(B_z\cup x_i)\sm\set{z}=A\cup\set{x_i}\in \cB$, contradiction. This completes the verification of \eqref{B|A}.

Now let $N_1=N-t_{i_0}\sim M$ and $M_1=N_1^{1/2}$ and  $i_1= M_1-1$ and let $t_i=t_{i_0}+\rdup{(i-i_0)M_1}+1$ for $i_0<i\leq i_0+i_1$.  Now going back to \eqref{B|A}? we see that for  $i_0<i\leq i_0+i_1$,
\[
\Pr(\cE_{\cB_{t_i}}\mid\cF_\cB)\leq 2Nk(n-k)^2 \brac{1-\frac{i_0M+(i-i_0)M_1}{N}}^3
\]
And so,
\[
\Pr(\exists i_0<i\leq i_0+i_1:\cE_{\cB_{t_i}}\mid t_{i_0+i_1}\leq \tau)=O\brac{i_1N(n-k)^2k\bfrac{M_1}{N}^3}=O\bfrac{n^3M_1^4}{N^2}=o(1).
\]
As before, there are at least 2 ``missing'' $k$-sets w.r.t. $B,B’$ and the probability they are both chosen by the time  $t_{i+1}$ is $O(M_1\cdot (n/(N_1-iM_1)^2)$. So, as before,  in \eqref{before},
\[
\Pr(\exists t_{i_0}< t\leq t_{i_0+i_1}:\cE_{\cB_t}\mid t_{i_0+i_1}\leq \tau )=o(1)+O\brac{\sum_{i=i_0+1}^{i_0+i_1}\frac{n^2M_1}{(N_1-(i-i_0)M_1)^2}}=o(1)+O\bfrac{n^2}{M_1}=o(1).
\]
Now $t_{i_0+i_1}=N-O(N_1^{1/2})$ and the probability that $\cE_{\cB}$ occurs between $t_{i_0+i_1}$ and  $\tau$ (see \eqref{B|A}) is bounded by 
\[
O\brac{N_1^{1/2}\cdot Nk(n-k)^2\bfrac{N_1^{1/2}}{N}^3} =o(1).
\] 
This completes the proof of Part (a) except for the question of $\cB$ defining a paving matroid. We deal with this in part (b). We claim that 
\beq{PCA}{
\lim_{n\to\infty} \Pr[\cF_\cB]=\begin{cases}0&c_n\to0.\\1-e^{-c^2}&c_n\to c.\\1&c_n\to \infty.\end{cases}
}
The event $\cF_\cB$ is monotone decreasing and so we only need consider the case $c_n\to c$. 
Let $Z$ be the number of (unordered) pairs $\set{A_1, A_2}\in \binom{[n]}{k}^2$ such that $|A_1\cap A_2|=k-1$ and $A_1,A_2\notin \cB$. Equivalently this is the number of sets $S\in \binom{[n]}{k+1}$ and unordered pairs $\set{x_1,x_2}$ such that $(S\sm\set{x_i})\notin \cB$. Call this condition $\neg\cB_{S,x_1,x_2}$. Then where $(b)_a=b(b-1)\cdots (b-a+1)$ for positive integer $a$ and for $\ell=O(1)$,
\begin{align}
\E[(Z)_\ell]&=\sum_{\substack{S_1,\ldots,S_\ell\in \binom{[n]}{k+1}\\|S_i\sm S_j|>2}}\brac{\frac{k(k+1)}{2}(1-p)^{2}}^\ell+X_\ell\label{6}\\
&\sim \brac{\binom{n}{k}\frac{k(n-k)}{2}(1-p)^2}^\ell+o(1)\nn\\
&\sim \bfrac{c^{2}}{2}^\ell.\label{7}
\end{align}
{\bf Explanation:} we choose a pair by choosing a set $S$ of size $k+1$ and choosing 2 elements for deletion. The sum in \eqref{6} comes from choosing $\ell$ sets $S_1,\ldots,S_\ell$ of size $k+1$ and choosing elements $x_{i,1},x_{i,2}\in S_i$ and taking $A_{i,j}=S_i\sm\set{x_{i,j}},j=1,2$. In the sum we insist that $|S_i\sm S_j|>2$ and this ensures that the $A_{i,j}$ are distinct. $X_\ell=o(1)$ is an error term to be discussed next. 

To bound $X_\ell$ we let $S_1,S_2,\ldots,S_r,r<\ell$ be a maximal set of $(k+1)$-sets that satisfy $|S_i\sm S_j|>2$. An $S_j,j>r$ that introduces a new $A_t$ introduces a factor $O(k^2n^2(1-p))=o(1)$ in place of $\binom{n}{k}k(n-k)(1-p)^2\sim c$. So, we can assume that for $j>r$, $S_j$ does not introduce any new non-edges. In which case $S_{r+1}$ creates non-edges $T_1,T_2$ where $T_i,i=1,2$ arises from $S_{r_i},i=1,2$. But this implies $|S_{i_1}\cap S_{i_2}|\geq k-1$, contradiction. This verifies $X=o(1)$ and yields \eqref{7}. The method of moments then implies that $Z$ is asymptotically Poisson with mean $c^2$. Thus \eqref{PCA} holds. 

We know from part (a) that w.h.p. we have $\neg \cE_{\cB}$ up until $\tau$. We also know that $\neg\cF_\cB$ implies $\cE_{\cB}$ deterministically. This finishes the proof of part (b) except for the claim with respect to paving matroids. For this we have to check whether there is a set of size $k-1$ that is not contained in a member of $\cB$. The probability of this can be bounded by
\[
\binom{n}{k-1}(1-p)^{n-k+1}=\binom{n}{k-1}\bfrac{c^2}{k(n-k)N}^{(n-k+1)/2}=o(1).
\]

\end{proof}

\section{Proof of Theorem \ref{th4}}
\subsection{The matching process}
In this subsection we prove Theorem \ref{th4} part \ref{th4a}) and the lower bound of part \ref{th4b}). These will follow from an analysis of the random greedy process for matchings in hypergraphs. 

We now describe the random greedy matching process. We initialize $\cH(0):=\cH$ and $V(0):=V(\cH)$, and proceed as follows. At step $i$ we choose an edge $e_i$ from $\cH(i)$ uniformly at random. We let $V(i+1):=V(i) \setminus e_i$ and $\cH(i+1) := \cH[V(i+1)]$. In other words, we form $\cH(i+1)$ by deleting the vertices of $e_i$ from $\cH(i)$. The process stops when $\cH(i)$ has no edges. We will argue that w.h.p. the process runs at least until step $i=i_{\max}$, where
\[
i_{\max}:=  \brac{1 - \phi^{\frac{1}{100k}}}\frac Nk.
\]

Let $t =t(i):= i / N$ and $p =p(t):= 1- kt$.  Note that $|V(i)|=Np(t)$. 
Note that at that step we would have $p(t(i_{\max})) = \phi^{1/100k}$. Indeed,
\beq{eqn:pbounds}{
\text{for all $i \le \imax$ we have} \quad \phi^{\frac{1}{100k}} \le p \le 1.
}
Let $D_v(i)$ be the set of neighbors of $v$ at step $i$, and $d_v(i):=|D_v(i)|$.  Define
\[
f=f(t) := 10 D\sqrt{\phi k \log N} p^{-10k}.
\]
\begin{claim}\label{clm:1}
    Let $\cE_{i'}$ be the event that for all $i \le i'$ and $v \in V(i)$ we have $|d_v(i) - Dp^{k-1}|\leq  f$. Then w.h.p. $\cE_{i_{\max}}$ holds. 
\end{claim}

\begin{proof}[Proof of Theorem \ref{th4} part \ref{th4a}) and the lower bound in part \ref{th4b}) given Claim \ref{clm:1}]
 Note that in the event $\cE_{\imax}$ the total number of edges remaining at any step $i \le \imax$ is
\[
Q(i) = \frac 1k \sum_{v \in V(i)} d_v(i) = \frac 1k \brac{ NDp^k \pm Np f}.
\]
In particular, we have for all $i \le \imax$ that 
\[
Q(i) \ge \frac {NDp^k}{k} \brac{ 1 - \frac{ f}{Dp^{k-1}}} >0
\]
since (using \eqref{eqn:th4assumption} and \eqref{eqn:pbounds})
\beq{errorbound}
{
\frac{ f}{Dp^{k-1}} = 10 \sqrt{k\phi \log N} p^{-11k+1} \le 10 \sqrt{k \log N} \phi^{\frac 12 - \frac {11}{100}} = o(1).
}
Since the process keeps running as long as $Q(i)>0$, part \ref{th4a}) holds. 

We move on to proving the lower bound in part \ref{th4b}). Let $\cS$ be the set of possible outcomes of the first $i_{\max}$ steps of the process such that $\cE_{i_{\max}}$ holds. In other words, $\cS$ is a collection of sequences of edges, with $i_{\max}$ edges per sequence, where each sequence forms a matching, and such that the event $\cE_{i_{\max}}$ holds when the process is given each sequence as input. For any sequence $S \in \cS$, the probability that the process produces $S$ (i.e. chooses the elements of $S$ in order to be the matching) is at most
\begin{align*}
    \prod_{i=1}^{i_{\max}} \frac{1}{\frac 1k \brac{ NDp^k - Np f}} & = \brac{\frac{k}{ND}}^{i_{\max}} \exp \brac{-\sum_{i=1}^{\imax} \log\brac{p^k - \frac{pf}{D}}}\\
    & \le \brac{\frac{k}{ND}}^{\imax} \exp \brac{-\sum_{i=1}^{\imax} k\log p + \frac{(1+o(1))f}{Dp^{k-1}}} \\
    & \le \brac{\frac{k}{ND}}^{\imax} \exp \brac{-kN \int_0^{1/k} \log(1-kt) \;  dt + o(N/k)}\\
     & \le \brac{\frac{k}{ND}}^{\imax} \exp \brac{N + o(N/k)} \\
     & = \brac{(1+o(1))\frac{ke^k}{ND}}^{\frac Nk}
\end{align*}
But by Claim \ref{clm:1}, the probability that the output of the process is in $\cS$ tends to 1, and so we must have
\begin{align*}
    |\cS| \ge \brac{(1+o(1))\frac{ND}{ke^k}}^{\frac Nk}.
\end{align*}
Now $\cS$ counts ordered matchings, and each unordered matching appears at most $\imax ! = \brac{(1+o(1))\frac{N}{ek}}^{N/k}$ times in $\cS$. Thus the number of matchings is at least 
\[
\bigg((1+o(1))De^{1-k}\bigg)^{\frac Nk},
\]
proving the lower bound in part \ref{th4b}).
\end{proof}

It remains to prove Claim \ref{clm:1}. The proof is a version of the proof by the first author and Bohman \cite{BB}, which has been adapted to handle a growing $k$ as well as simplified since at the moment we do not care to optimize error terms. We apply the so-called differential equation method. The unfamiliar reader can see the introduction to the method by the first author and Dudek \cite{BD}, especially Section 5 which is essentially the case $k=2$ of the following proof.

\begin{proof}[Proof of Claim \ref{clm:1}]
    
 Let
\[
Z_v(i):= \begin{cases}
d_v(i)-Dp^{k-1}-f(t) & \text { if } \mathcal{E}_{i-1} \text { holds } \\ 
Z_v(i-1) & \text { otherwise. }
\end{cases}
\]
We will now verify that the sequence $(Z_v(i))_{i \ge 0}$ is a supermartingale with respect to the natural filtration $(\cF_i)_{i \ge 0}$. First note that if $\cE_i$ fails then $\Delta Z_v(i)=0$ by definition (and so the supermartingale condition holds in this case), so we assume  $\cE_i$ holds.
We have 
\begin{align}
\E[\Delta d_v(i) | \cF_i] &= -\frac{1}{Q} \sum_{\substack{E \in D_v(i) \\ v' \in E \setminus \{v\}}} d_{v'} + O\brac{\frac{k^2 D \phi d_v}{Q}}\label{eqn:deltad1}\\
& \le  -\frac{k(k-1)\brac{Dp^{k-1}-f}^2}{NDp^k + Np f }+O\brac{\frac{k^3 D \phi}{Np}}\label{eqn:deltad2}\\
& \le -k(k-1)p^{k-2} \frac{D}{N} + 4k^2 p^{-1} f \frac{1}{N}+O\brac{\frac{k^3 D \phi}{Np} }\label{eqn:deltad3}
\end{align}
{\bf Explanation:} Given  there are $Q$ equally likely possibilities for $e_i$. Summing the number of edges removed from $D_v$ over all possibilies for $e_i$ gives 
\[
\sum_{\substack{E \in D_v(i) \\ v' \in E \setminus \{v\}}} (d_{v'}-1) + O\brac{k^2 D \phi d_v},
\]
where the big-O term accounts for overcounting. This justifies \eqref{eqn:deltad1}. Now \eqref{eqn:deltad2} follows from our estimates of $d_v, d_{v'}$ and $Q$. 
In particular, note that due to \eqref{errorbound}, for all $i \le \imax$ we have
\beq{est}
{
d_v = Dp^{k-1}\brac{1 \pm \frac{f}{Dp^{k-1}}} = (1+o(1))Dp^{k-1}, \qquad Q = \frac {NDp^k}{k} \brac{1 \pm \frac{f}{Dp^{k-1}}} = (1+o(1))\frac {NDp^k}{k}
}
Finally, to justify \eqref{eqn:deltad3} we note that 
\[
\frac{k(k-1)\brac{Dp^{k-1}-f}^2}{NDp^k + Np f } = k(k-1)p^{k-2} \frac{D}{N} \cdot \frac{\brac{1 - \frac{f}{Dp^{k-1}}}^2}{1 + \frac{f}{Dp^{k-1}} }.
\]
Now by Taylor's theorem we can estimate the one-step change of the deterministic part of $Z_v(i)$:
\beq{deltadeterministic}
{
    \Delta \brac{Dp^{k-1} + f} = -k(k-1)p^{k-2}\frac DN + f' \frac 1N + O\brac{\frac{k^4 D + f''}{N^2}},
}
where the $f''$ above is evaluated at some value between $ t=i/N$ and $ (i+1)/N$.
Now we have
\begin{align}
    \E[\Delta Z_v(i) | \cF_i] & \le \brac{4k^2 p^{-1}f - f'}N^{-1} + O\brac{\frac{k^3 D \phi}{Np} +\frac{k^4 D + f''}{N^2}}\label{eqn:deltaz1}\\
    & = -\frac{60 k^{5/2}D\sqrt{\phi \log N}}{Np^{10k+1}} + O\brac{\frac{k^3 D \phi}{Np}  +\frac{k^4 D }{N^2}+\frac{ k^{\frac 92} D\sqrt{\phi \log N}}{N^2p^{10k+2}}}\label{eqn:deltaz2}
\end{align}
{\bf Explanation:} \eqref{eqn:deltaz1} follows from \eqref{eqn:deltad3} and \eqref{deltadeterministic}.  \eqref{eqn:deltaz2} follows from evaluating $f', f''$. 

Finally, to verify that $(Z_v(i))_{i \ge 0}$ is a supermartingale we check that the negative first term in \eqref{eqn:deltaz2} dominates big-O terms. We will use the fact that we always have $1 \ge p \ge p(t(\imax)) = \phi^{1/100k}$. The ratio of the big-O terms to the first term is on the order
\begin{align*}
\frac{\frac{k^3 D \phi}{Np}  +\frac{k^4 D }{N^2}+\frac{ k^{\frac 92} D\sqrt{\phi \log N}}{N^2p^{10k+2}}}{\frac{k^{5/2}D \sqrt{\phi \log N}}{Np^{10k+1}}} \le  \bfrac{k\phi}{\log N}^{1/2}   + \frac{k^{3/2}}{N}\cdot \phi^{-1/2}  + \frac{k^2}{Np} =o(1)
\end{align*}
(using \eqref{eqn:th4assumption}, \eqref{eqn:phibound} and \eqref{eqn:pbounds}).

\begin{lemma}[Freedman]\label{lem:FreedmanLemma} Let $Y(i)$ be a supermartingale with respect to the filtration $(\cF_i)_{i \ge 0}$. Suppose $|\Delta Y(i)| \leq C$ for all $i$, and 
\[
W(i)\coloneqq \sum_{k \leq i} \operatorname{Var}\left[\Delta Y(k) \mid \mathcal{F}_k\right].
\]
Then
$$
\mathbb{P}[\exists i: W(i) \leq w, Y(i)-Y(0) \geq d] \leq \exp \left(-\frac{d^2}{2(w+C d)}\right).
$$
\end{lemma}
We apply the above lemma to the supermartingale $(Z_v(i))_{i \ge 0}$. By the triangle inequality we have
\beq{eqn:absbound}
{
|\Delta Z_v(i)| \le |\D d_v(i)|+ |\D (Dp^{k-1}+f(t))| \le kD \phi + \frac{2k^2 D}{N} \le 3kD \phi,
}
where we have used \eqref{deltadeterministic} (the ``2'' justifies dropping the lesser order terms) and the  bound \eqref{eqn:th4assumption}.
So we take $C = 3kD \phi$. Now note that 
\beq{eqn:expbound}
{
\E[|\Delta Z_v(i)| \; | \; \cF_i] \le \E[|\Delta d_v(i)| \; | \; \cF_i] + |\D (Dp^{k-1}+f(t))| \le \frac{3k^2 D}{N}.
}
Thus, since $|\Delta Z_v(i)| \le C$,
\[
Var [\Delta Z_v(i) | \cF_i] \le C\E[|\Delta Z_v(i)| \; | \; \cF_i]\le \frac{9 k^2D^2 \phi}{N}.
\] 
We take $w = 10kD^2 \phi$, and $d= 10 D\sqrt{k\phi \log N}$. Note that $Z_v(0) = -f(0) =  -d$. Thus Freedman's inequality gives us
\begin{align*}
   \mathbb{P}[\exists i \le \imax :  Z_v(i) \ge 0] &\leq \exp \left(-\frac{d^2}{2(w+C d)}\right)\\
   & = \exp \left(-\frac{100 D^2 k\phi \log N}{20kD^2 \phi+6kD\phi \cdot 10 D\sqrt{k\phi \log N}}\right)\\
   & = \exp \left(-\frac{100 \log N}{20+60 \sqrt{\phi k\log N}}\right)\\
   & = N^{-5+o(1)},
\end{align*}
small enough to beat the union bound over $N$ choices for $v$. Thus, w.h.p. we have $Z_v(i) <0$ for all $i \le \imax$. This proves that $d_v(i) \le Dp+f$. The fact that $d_v(i) \ge Dp-f$ w.h.p. follows from a symmetric argument. 

\end{proof}

\subsection{Entropy}

In this section we prove the upper bound for part \ref{th4b}) of Theorem \ref{th4}.

\begin{lemma}\label{lem:perfmatch}
    Let $\cH$ be a $k$-uniform hypergraph with $n$ vertices and average degree $D$. Suppose every pair of vertices is in at most $D\phi$ edges together. Then $\cH$ has at most
    \[
    \brac{\brac{1+O\brac{k\phi^\frac{1}{k-1}}} De^{1-k}}^{N/k}
    \]
    perfect matchings.
\end{lemma}

We refer the unfamiliar reader to the book by Cover and Thomas \cite{CT} for an introduction to entropy. 

\begin{proof}
    Let $\cX$ be the collection of all perfect matchings in $\cH$, and let $X \in \cX$ be chosen uniformly at random. Let $\lambda:=V\rightarrow [0, 1]$ be uniform, i.e. for each $v \in V$ we choose $\lambda(v) \in [0,1]$ uniformly and independently. For each $v \in V$ we let $X_v$ be the edge of the matching $X$ which covers $v$. We let $X_v^{\lambda}$ be the set of edges $\{X_{v'}: \lambda(v') > \lambda(v)\}$. We imagine revealing the edges of $X$ by revealing $X_v$ for each vertex $v$ in order of decreasing $\lambda$ value. Of course, when we reach $v$ we might already know $X_v$, which happens if some other vertex $u\in X_v$ has $\lambda(u) > \lambda(v)$. For any fixed $\lambda$, the entropy of $X$ is
\beq{eqn:entropy1}
{
H(X)=  \log(|\cX|)    = \sum_{v \in V} H(X | X_v^\l),
}
where the first equality is because $X$ is uniform, and the second is the chain rule for entropy. For each $v$, we let $N_v$ be the number of ``possible'' (explained momentarily) edges that could be $X_v$ if we already know $X_v^\lambda$. If $v$ is not the first vertex in $X_v$ to be revealed, then we already know $X_v$ and so the number of possibilities is $N_v=1$. Otherwise, $N_v$ is the number of edges of $\cH$ containing $v$ and which do not intersect with $X_{v'}$ for any $v'$ with $\lambda(v')>\lambda(v)$. Note that $N_v$ depends on $\lambda$. Then from \eqref{eqn:entropy1} we get
\beq{eqn:entropy2}
          {
           \log(|\cX|) = \sum_{v \in V} H(X | X_v^\l) \le \sum_{v \in V} \E_X[\log(N_v)] ,
}
where the  inequality is because the entropy of any random variable is at most the log of its support. The above is true for any fixed $\lambda$, so it remains true if we take the expected value of both sides over all possible $\lambda$. We get
\beq{eqn:entropy3}
          {
           \log(|\cX|) \le  \E_\lambda \sum_{v \in V} \E_X[\log(N_v)] = \E_X  \sum_{v \in V} \E_{\l(v)}  \E_{\l }[\log(N_v) | \l(v)],
}
where in the final expression (starting with the innermost expectation), we take the expectation over our choice of values $\lambda(v'), v' \neq v$ given some fixed $X$ and $\lambda(v)$. Then we take the expectation over $\lambda(v)$, sum this over $v \in V$, and take the expectation over the random matching $X$.
    
    We let $\cF_v$ be the event that $v$ comes first among all vertices of $X_v$. We let $B=B(X) \subseteq E(\cH)$ be the  set of edges $e$ such that some edge of $X$ shares at least two vertices with $e$. We call the edges counted by $B$ {\em bad}. For each $v \in V$ let $B_v$ be the set of bad edges containing $v$. Note that 
    \beq{eqn:bbound}
    {
    |B| \le |X| \binom k2 D\phi \le knD\phi.
    }
      We bound the innermost expectation of \eqref{eqn:entropy3}.
\begin{align*}
     \E_{\l }[\log(N_v) | \l(v)]  &= \l(v)^{k-1} \E_{\l }[\log(N_v) | \l(v), \cF_v]\\
     & \le \l(v)^{k-1}  \log\brac{\E_{\l }[N_v | \l(v), \cF_v]}\\
     & \le \l(v)^{k-1} \log \brac{1 + d_v\l(v)^{k(k-1)} +  |B_v|}.
\end{align*}
{\bf Explanation:} The first line follows since $N_v=1$ whenever $\cF_v$ fails, and the probability of $\cF_v$ is $\l(v)^{k-1}$. The second line is Jensen's inequality. To get the last line, for each edge containing $v$ we bound the probability that this edge is counted by $N_v$. Of course $X_v$ is counted, explaining the ``1'' inside the logarithm. For any bad edge we trivially bound the probability it is counted by 1, explaining the ``$|B_v|$''. Now for any edge $E=\{v=v_1, v_2, \ldots, v_k\}$ that is not bad,  the edge  counts for $N_v$ only if all the $k(k-1)$ vertices in $X_{v_2}, \ldots, X_{v_k}$ appear after $v$ in the ordering $\lambda$, which happens with probability $\lambda(v)^{k(k-1)}$.

   Now taking the expectation $\E_{\l(v)}$ of the above we get the bound 
    \begin{align*}
     \E_{\l(v)} \E_{\l }[\log(N_v) | \l(v)]  & \le \E_{\l(v)} \left[ \l(v)^{k-1} \log \brac{1 + d_v\l(v)^{k(k-1)} +  |B_v|}\right]\\
     & = \int_0^1 x^{k-1} \log\brac{1+d_vx^{k(k-1)} +  |B_v|} \; dx.
\end{align*}
Summing over $v$ we get
\begin{align}
     \sum_{v \in V} \E_{\l(v)} \E_{\l }[\log(N_v) | \l(v)]  & \le \sum_{v \in V} \int_0^1 x^{k-1} \log\brac{1+d_vx^{k(k-1)} + |B_v|} \; dx\nonumber\\
     & \le N \int_0^1 x^{k-1} \log\brac{1+Dx^{k(k-1)} +kD\phi} \; dx\nonumber\\
     & \le  \frac Nk \int_0^1 \log D +  \log\brac{u^{k-1} + 2k\phi} \; du\nonumber\\
     & \le  \frac Nk \int_0^1 \log D +  (k-1)\log\brac{u+ \brac{2k\phi}^{\frac{1}{k-1}}} \; du\label{eqn:entropy4}
\end{align}
{\bf Explanation:} The second line follows by convexity of $\log(x)$. On the third line we have used that $1 \le kD \phi$ and substituted $u = x^k$. On the fourth line we have used 
\[
u^{k-1} + 2k\phi \le \brac{u+  \brac{2k\phi}^{\frac{1}{k-1}}}^{k-1}.
\]
Now to evaluate the integral in \eqref{eqn:entropy4}, we have
\begin{align*}
    \int_0^1 \log\brac{u+  \brac{2k\phi}^{\frac{1}{k-1}}} \; du &= \int_{ \brac{2k\phi}^{\frac{1}{k-1}}}^{1+ \brac{2k\phi}^{\frac{1}{k-1}}} \log w \; dw\\
    & = \Big[w \log w - w \Big]_{ \brac{2k\phi}^{\frac{1}{k-1}}}^{1+ \brac{2k\phi}^{\frac{1}{k-1}}}\\
    & = -1 + O\brac{\phi^{\frac{1}{k-1}}}.
\end{align*}
Returning to \eqref{eqn:entropy3} and using \eqref{eqn:entropy4} (and evaluating the integral as above), we have
\beq{entropy5}
{
\log(|\cX|) \le \E_X \frac Nk \brac{ \log D - k + 1 + O\brac{k \phi^{\frac{1}{k-1}}}} = \frac Nk \brac{ \log D - k + 1 + O\brac{k \phi^{\frac{1}{k-1}}}}
}
where the last equality is just because there is no dependence on $X$ anymore. 
\end{proof}

\begin{proof}[Proof of the upper bound in Theorem \ref{th4} part \ref{th4b})]
    Let $V' \subseteq V$ with $|V'|=N' \le N$. $\cH[V']$ is a $k$-uniform hypergraph with average degree at most $D$. 
So by Lemma \ref{lem:perfmatch} there are at most 
    \[
    \brac{\brac{1+O\brac{\phi^\frac{1}{k-1}}} De^{1-k}}^{\frac{N'}{k}} 
    \]
    perfect matchings of $\cH[V']$.  Summing over choices of $V'$ and possible perfect matchings of $\cH[V']$ we see that the number of matchings in $\cH$ is at most 
       \begin{align*}
            \sum_{N' = 0}^N \binom {N}{N'} \brac{\brac{1+O\brac{\phi^\frac{1}{k-1}}} De^{1-k}}^{\frac{N'}{k}} &= \brac{1 + \brac{\brac{1+O\brac{\phi^\frac{1}{k-1}}} De^{1-k}}^{\frac 1k}}^N\\
            & = \brac{\brac{1+O\brac{\phi^\frac{1}{k-1}}} De^{1-k}}^{\frac Nk}.
       \end{align*}

\end{proof}

   \section{Proof of Theorems \ref{th2} and \ref{th3}}      

        \begin{proof}[Proof of Theorem \ref{th2}]
            As we discussed in the introduction, there is a one-to-one correspondence between sparse paving matroids of rank $k$ on $[n]$ and partial Steiner systems $S_p(n, k, k-1)$. Meanwhile, these partial Steiner systems correspond to matchings in the following hypergraph. Let $\cH$ be the hypergraph with
            \[
            V(\cH) = \binom{[n]}{k-1}, \qquad E(\cH) = \left\{\binom{K}{k-1}: K \in \binom{[n]}{k}\right\}.
            \]
           We apply Theorem \ref{th4} part \ref{th4b}. $\cH$ is $k$-uniform and has $N = \binom{n}{k-1}$ vertices. $\cH$ is $D$-regular with $D = n-k+1$. Any pair of vertices is in at most one edge altogether, so we can take $\phi = 1/D$. By taking $k = o(\log n)$ we satisfy \eqref{eqn:th4assumption}. The theorem follows. 
        \end{proof}

      Next we prove Theorem \ref{th3}. The lower bound for our estimate follows from Theorem \ref{th2}. The upper bound for our estimate will follow from results in \cite{PV, HPV}. In those papers $k$ was fixed, so here we just have to check how fast $k$ can grow. Let
          \beq{eqn:f45}
          {
          f^{(4.5)}_k(n): = \begin{cases}
             3k\brac{\frac{8k}{n} + \frac{1}{n-k+1}}^\frac{1}{k-1} + k \frac{\log(e\sqrt{\log n})}{\sqrt{\log n}}+ \frac{3}{\log n}, & n \ge n_0\\
             C_k, & n < n_0
          \end{cases}
          }
          where $C_k$ depends only on $k$, and $n_0=n_0(k)$ is large enough so that for all $n \ge n_0$ we have
          \beq{eqn:n0}
          {
          \log(e(n-k+1)\log^2(n))\le \frac{k-1}{2} \log(e^{1-k}(n-k+1)).
          }
           In \cite{HPV} they show the following, which will help us bound $p(n, k)$.
      \begin{lemma}[Lemmas 4.4 and 4.5 \cite{HPV}]\label{lem:hpv}
          Let $k \ge 4$ and $n \ge \exp\brac{ \brac{\frac{(k-1)(k+1)}{k-3}}^2}$. Then 
          \beq{eqn:pnkbound}
          {
          p(n, k) = \sum_{a=0}^{\binom{n}{k-1}} \sum_{b=0}^{\binom{n}{k-1}} p(n, k, a, b)
          }
          where
           \beq{eqn:pnkabbound}
          {
          \log p(n, k, a, b) \le \frac1k \binom{n}{k-1}\brac{\log(n-k+1) +1-k+f_k^{(4.5)}(n)}
          }
          
      \end{lemma}
In \cite{PV} they showed the following. Once we bound $p(n, k)$ we will use this to bound $m(n, k)$. 
\begin{lemma}[Lemma 21 in \cite{PV}]\label{lem:pv}
    $m(n, k) \le m(n, k-1)p(n, k).$
\end{lemma}
      
      \begin{proof}[Proof of Theorem \ref{th3}] Since $m(n, k) \ge p(n, k) \ge s(n, k)$, the lower bounds on $\log m(n, k)$ and $\log p(n, k)$ follow from Theorem \ref{th2}. 
      \eqref{eqn:pnkbound} and \eqref{eqn:pnkabbound} give
      \begin{align}
           \log p(n, k) &\le \log\brac{\binom{n}{k-1}^2 \max_{a, b} \{p(n, k, a, b)\}}\nonumber\\
           &\le 2 \log \binom{n}{k-1} + \frac1k \binom{n}{k-1}\brac{\log(n-k+1) +1-k+f_k^{(4.5)}(n)}\label{eqn:39}\\
           & =  \frac1n \binom{n}{k}\Big(\log n +1-k+o(1)\Big),\label{eqn:40}
      \end{align}
     where the last line is justified as follows. The first term in \eqref{eqn:39} is absorbed by the $o(1)$ term in \eqref{eqn:40}. Also, $\log(n-k+1) = \log n + o(1)$. Finally we claim that $f_k^{(4.5)}(n)=o(1)$. Indeed, first note that as $k = o(\log^{1/2}n / \log \log n)$ and  $n \rightarrow \infty$ we have $n \ge n_0(k)$ since $\eqref{eqn:n0}$ holds. Furthermore, examining the first line of \eqref{eqn:f45}, we see that $f_k^{(4.5)}(n)=o(1)$. 

      Now we bound $m(n, k)$. Note that $m(n, 0)=1$ and so, applying Lemma \ref{lem:pv} iteratively we have
      \begin{align*}
          \log m(n, k) &\le \log \big(p(n, k) p(n, k-1) \cdots p(n, 1)\big)\\
          & = \sum_{j=1}^k \log p(n, j)\\
          & =  \frac1n \binom{n}{k}\Big(\log n +1-k+o(1)\Big),
      \end{align*}
      where the last line follows from using \eqref{eqn:40} to bound $\log p(n, k)$ and absorbing all other terms into the $o(1)$. This completes the proof of Theorem \ref{th3}.
  \end{proof}

{\bf Acknowledgement:} we thank Jasper Dekoninck and Tim Gehrunger for pointing to an error in part (b) of Theorem \ref{th1}.
\bibliographystyle{abbrv}
\bibliography{refs.bib}

\end{document}